\documentclass[11pt]{amsart}
\newtheorem{theoreme}{Theorem}
\newtheorem{fait}[theoreme]{Fact}
\newtheorem{lemme}[theoreme]{Lemma}
\newtheorem{proposition}[theoreme]{Proposition}

\newtheorem{corollaire}[theoreme]{Corollary}

\newtheorem{conjecture}[theoreme]{Conjecture}

\newtheorem{thmABC}{Theorem}
\newtheorem{factABC}{Fact}
\newtheorem{lemABC}{Lemma}

\theoremstyle{definition}
\newtheorem{remarque}[theoreme]{Remark}

    \usepackage{mathrsfs}

\usepackage{xy}
\xyoption{all}
\usepackage[normalem]{ulem}
\usepackage{enumerate}
\usepackage[a4paper, left=1.5cm,right=1.5cm,bottom=2cm,top=2cm]{geometry}
\usepackage{xcolor}
\usepackage{graphicx,epsfig}
\usepackage{amscd}
\usepackage{amssymb}
\usepackage[shortlabels]{enumitem}
\usepackage{ulem}
\usepackage{hyperref}
\usepackage[utf8]{inputenc}
\usepackage[T1]{fontenc}
\usepackage{lmodern}
\usepackage{tensor}
\usepackage{colonequals}
\usepackage{amsxtra}

\hypersetup{colorlinks,allcolors=blue}

      \newcommand{\crys}{{\hbox{\rm \tiny crys}}}
   
   \newcommand{\cpt}{{\hbox{\rm \tiny cpt}}}

       \newcommand{\ngen}{{\hbox{\rm \tiny ngen}}}
   \newcommand{\gen}{{\hbox{\rm \tiny gen}}}
   \newcommand{\topo}{{\hbox{\rm \tiny top}}}
   \newcommand{\et}{{\hbox{\rm \tiny et}}}
   \newcommand{\an}{{\hbox{\rm \tiny an}}}
   
     \newcommand{\tors}{\hbox{\rm \tiny tors}}
            \newcommand{\free}{\hbox{\rm \tiny free}}

  \newcommand{\Sim}{\hbox{\rm im}}

 \newcommand{\SH}{\hbox{\rm H}}
  
  \newcommand{\SGL}{\hbox{\rm GL}}

                              \newcommand{\CH}{\hbox{\rm CH}}

     \newcommand{\C}{\mathbb{C}}
    \newcommand{\F}{\mathbb{F}}
      
         \newcommand{\Q}{\mathbb{Q}}
           
     \newcommand{\Z}{\mathbb{Z}}

\newcommand{\cS}{\mathcal{S}}

\newcommand{\cZ}{\mathcal{Z}}

\newcommand{\cV}{\mathcal{V}}
\newcommand{\cX}{\mathcal{X}}

\newcommand{\Spec}{\mathrm{Spec}}

\newcommand{\Ob}{\mathrm{Ob}}

\title{Uniform bounds for obstructions to the integral  Tate conjecture}

\author[Anna Cadoret]{Anna Cadoret}
\address{\hspace{-0.5cm}Anna Cadoret\newline
IMJ-PRG -- Sorbonne Universit\'e, Paris, FRANCE\newline
\textit{anna.cadoret@imj-prg.fr}}

\author[Alena Pirutka]{Alena Pirutka}
\address{\hspace{-0.5cm}Alena Pirutka \newline
CIMS,
New York University, 
New York, U.S.A.\newline 
\textit{pirutka@cims.nyu.edu}}

\date{\today}
\begin{document}
\maketitle

\begin{abstract}
Assuming natural variational realization conjectures, we give uniform bounds for the obstruction to the integral  Tate conjecture   in $1$-dimensional families of algebraic varieties over an infinite finitely generated field.
 
  \begin{center} 2020 \textit{Mathematics Subject Classification.} Primary: 14C25; Secondary:  14D07, 14D10, 14F30, 14G25.\end{center}
 \end{abstract} 

\section{Introduction}\textit{}

\noindent  For an abelian group $A$,  write $A_{\tors}\subset A$ for its torsion subgroup and $A\twoheadrightarrow A^{\free}:=A/A_{\tors}$ for its maximal torsion-free quotient. For an algebraic group $G$, let $G^\circ\subset G$ denote its neutral component and $G\twoheadrightarrow \pi_0(G):=G/G^\circ$ its group of connected components.\\

\noindent A variety over a field $k$ is a separated scheme  of finite type over $k$.\\

\noindent  In this paper $k$ will denote an infinite field of characteristic $p\geq 0$, finitely generated over its prime subfield. We fix  a separable closure $k\hookrightarrow \bar k$ and write $\pi_1(k)=$Gal$(\bar k|k)$ for the absolute Galois group. \\

\subsection{Tate conjectures}Let $X$  be a smooth projective variety  over $k$. For every  integer  $i\geq 0$, let $\CH^i(X)$ denote the group of algebraic cycles of codimension $i$  on $X$ modulo rational equivalence,  and for every ring $R$, set $\CH^i(X)_R:=\CH^i(X)\otimes_{\Z}R$. For a prime $\ell\not= p$, set  $$V_{\Z_\ell}:=\SH^{2i}(X_{\bar k},\Z_\ell(i)).$$ Let $G_\ell\subset \SGL(V_{\Q_\ell})$ denote  the Zariski-closure of the image of   $\pi_1(k)$   acting on $V_{\Q_\ell}:=V_{\Z_\ell}\otimes_{\Z_\ell} \Q_\ell $ and let 
$$\widetilde{V}_{\Q_\ell}:=(V_{\Q_\ell})^{G_\ell^\circ}\subset V_{\Q_\ell}$$ 
denote  the  $\Q_\ell$-vector space of Tate classes.   The cycle class map   $c_\ell: \CH^i(X_{\bar k})\rightarrow   V_{\Z_\ell} $ 
for $\Z_\ell$-\'etale cohomology fits into the following canonical Cartesian diagram
 \begin{equation}\label{Diagram:Basic}\xymatrix{\CH^i(X_{\bar k}) \ar[r]\ar[dd]\ar@/^2pc/[rrrr]^{c_\ell}&\CH^i(X_{\bar k}) _{\Z_\ell}\ar@{->>}[r]\ar[dd]&V_{\Z_\ell}^a\ar@{^{(}->}[r]\ar@{->>}[d]&\widetilde{V}_{\Z_\ell}\ar@{^{(}->}[r]\ar@{->>}[d]\ar@{}[dr]|\square&V_{\Z_\ell}\ar@{->>}[d]\\
&&V_{\Z_\ell}^{\free,a}\ar@{^{(}->}[r]\ar@{^{(}->}[d]&\widetilde{V}^{\free}_{\Z_\ell}\ar@{^{(}->}[r]\ar@{^{(}->}[d]\ar@{}[dr]|\square&V_{\Z_\ell}^{\free}\ar@{^{(}->}[d]\\
\CH^i(X_{\bar k})_\Q\ar[r]&\CH^i(X_{\bar k})_{\Q_\ell}\ar@{->>}[r]&V_{\Q_\ell}^a\ar@{^{(}->}[r]&\widetilde{V}_{\Q_\ell}\ar@{^{(}->}[r]&V_{\Q_\ell},}\end{equation}
where $V_{\Z_\ell}^a$ (resp. $V_{\Q_\ell}^a$) is the image of the cycle class map $c_\ell\otimes \Z_\ell: \CH^i(X_{\bar k})_{\Z_\ell}\rightarrow   V_{\Z_\ell} $ (resp. $c_\ell\otimes \Q_\ell$) and where $\widetilde{V}_{\Z_\ell}$ and $\widetilde{V}^{\free}_{\Z_\ell}$ are defined by the rightmost Cartesian squares of the diagram.\\

\noindent The (classical) rational $\Q_\ell$-Tate conjecture
for  codimension $i$ cycles on   $X$ \cite{TateWH} 
$$ \hbox{\rm Tate}_{\Q_\ell}(X,i) \;\; 
   V_{\Q_\ell}^a =\widetilde{V}_{\Q_\ell}$$
admits the following integral variants:
 $$\begin{tabular}[t]{lll}
$ \hbox{\rm Tate}^{\free}_{\Z_\ell}(X,i)$&
   $V_{\Z_\ell}^{\free,a} =\widetilde{V}^{\free}_{\Z_\ell}$& (Integral Tate conjecture modulo torsion);\\
 $\hbox{\rm Tate}_{\Z_\ell}(X,i)$& 
$V_{\Z_\ell}^a =\widetilde{V}_{\Z_\ell}$&  (Integral Tate conjecture).
\end{tabular}$$
\noindent While, tautologically, 
 $$\hbox{\rm Tate}_{\Z_\ell}(X,i)\Rightarrow  \hbox{\rm Tate}^{\free}_{\Z_\ell}(X,i)\Rightarrow  \hbox{\rm Tate}_{\Q_\ell}(X,i),$$
it is known that, in general, the converse implications fail (see e.g. \cite{CTS, AH} for an example of the failure of $\hbox{\rm Tate}_{\Z_\ell}(X,i)$ and \cite{CTS, Kollar, Totaro} for examples of the failure of $\hbox{\rm Tate}^{\free}_{\Z_\ell}(X,i)$).\\

\noindent The aim of this note is to analyze the obstructions to $\hbox{\rm Tate}_{\Z_\ell}(X,i)$, $\hbox{\rm Tate}^{\free}_{\Z_\ell}(X,i)$ when $X$ varies in  family. 
Our arguments provide a new application of the structure theorem of the degeneration locus of $\ell$-adic local systems of \cite{UOI2} (see Fact \ref{GenericPts:Spar1}),  in the spirit of \cite{CC, Bas}.\\

\noindent Before considering the variational setting, we make some elementary remarks.  By definition, the obstructions to  $ \hbox{\rm Tate}_{\Q_\ell}(X,i)$,  $ \hbox{\rm Tate}^{\free}_{\Z_\ell}(X,i)$, $ \hbox{\rm Tate}_{\Z_\ell}(X,i)$ are, respectively:

$$\widetilde{C}_{\Q_\ell} := \widetilde{V}_{\Q_\ell}/V_{\Q_\ell}^a,\;\;  \widetilde{C}^{\free}_{\Z_\ell} := \widetilde{V}^{\free}_{\Z_\ell}/V_{\Z_\ell}^{\free,a},\;\;  \widetilde{C}_{\Z_\ell}:= \widetilde{V}_{\Z_\ell}/V_{\Z_\ell}^a. $$

\subsubsection{$ \widetilde{C}^{\free}_{\Z_\ell}$ \textit{versus}  $\widetilde{C}_{\Z_\ell}$} The short exact sequence 
\begin{equation}\label{Eq:Obstruction1}0\rightarrow (V_{\Z_\ell})_{\tors}/(V_{\Z_\ell}^a)_{\tors}  \rightarrow\widetilde{C}_{\Z_\ell} \rightarrow  \widetilde{C}^{\free}_{\Z_\ell}\rightarrow 0\end{equation}
realizes  $\widetilde{C}_{\Z_\ell}$  an extension of $ \widetilde{C}^{\free}_{\Z_\ell}$ by a finite group which is a quotient of $(V_{\Z_\ell})_{\tors}$. As $(V_{\Z_\ell})_{\tors}$ is constant in  family, the problems of bounding uniformly   $\widetilde{C}^{\free}_{\Z_\ell}$ and  $\widetilde{C}_{\Z_\ell}$ are essentially equivalent.

\subsubsection{$ \widetilde{C}_{\Q_\ell}$ \textit{versus}  $\widetilde{C}^{\free}_{\Z_\ell}$} From $ \widetilde{C}_{\Q_\ell}= \widetilde{C}^{\free}_{\Z_\ell}\otimes_{\Z_\ell}\Q_\ell$ and the short exact sequence (\ref{Eq:Obstruction1}), one has the folllowing equivalences $$\hbox{\rm Tate}_{\Q_\ell}(X,i) \Leftrightarrow (\widetilde{C}^{\free}_{\Z_\ell})_{\tors}= \widetilde{C}^{\free}_{\Z_\ell} \Leftrightarrow  (\widetilde{C}_{\Z_\ell})_{\tors}= \widetilde{C}_{\Z_\ell}$$
and, in case they hold, (\ref{Eq:Obstruction1}) reads 
\begin{equation}\label{Eq:Obstruction2}0\rightarrow (V_{\Z_\ell})_{\tors}/(V_{\Z_\ell}^a)_{\tors}  \rightarrow(\widetilde{C}_{\Z_\ell})_{\tors} \rightarrow  (\widetilde{C}^{\free}_{\Z_\ell})_{\tors}\rightarrow 0.\end{equation}
So that, assuming $ \hbox{\rm Tate}_{\Q_\ell}(X,i)$, the obstructions we are interested in are $(\widetilde{C}_{\Z_\ell})_{\tors}$, $ (\widetilde{C}^{\free}_{\Z_\ell})_{\tors}$. The obstruction  $ (\widetilde{C}^{\free}_{\Z_\ell})_{\tors}$ can be described without involving the $\Z_\ell$-module $\widetilde{V}^{\free}_{\Z_\ell}$ of Tate classes. Indeed, writing $$C^{\free}_{\Z_\ell}:=V^{\free}_{\Z_\ell}/V^{\free,a}_{\Z_\ell},$$
 it follows from the short exact sequence 
 $$0\rightarrow C^{\free}_{\Z_\ell}\rightarrow \widetilde{C}^{\free}_{\Z_\ell}\rightarrow V^{\free}_{\Z_\ell}/\widetilde{V}^{\free}_{\Z_\ell}\rightarrow 0$$ 
 and the fact that $ V^{\free}_{\Z_\ell}/\widetilde{V}^{\free}_{\Z_\ell}$ is torsion-free that 
 $$(C^{\free}_{\Z_\ell})_{\tors}=(\widetilde{C}^{\free}_{\Z_\ell})_{\tors}.$$

 \subsection{}\label{Sec:Statements}\hspace{-0.3cm}Let now $S$ be a smooth, geometrically connected variety over $k$, with generic point $\eta$, and $f:X\rightarrow S$ a smooth projective morphism. For $s\in S$, denote by a subscript $(-)_s$ the various modules attached to $X_s$ introduced above (\textit{e.g.} $V_{\Z_\ell,s}:=\SH^{2i}(X_{\bar s},\Z_\ell(i))$, $V^a_{\Z_\ell,s}:=\hbox{\rm im}[\CH^i(X_{\bar s})_{\Z_\ell}\rightarrow V_{\Z_\ell,s}]$ \textit{etc.}).  One would like to investigate how 
 $$\widetilde{\Ob}_{\Z_\ell,s}:=|(\widetilde{C}_{\Z_\ell,s})_{\tors}|$$
 vary with $s\in |S|$. 
 \noindent In particular, the vanishing of the obstruction group $(\widetilde{C}_{\Z_\ell,s})_{\tors}$ reads as  $\widetilde{\Ob}_{\Z_\ell,s}=1$.

 \subsubsection{}\label{Sec:MainThm1}\hspace{-0.3cm}Assume first $p=0$. The following statement is predicted by the main conjecture of \cite{Bas}. For every integer $d\geq 1$, let  $|S|^{\leq d}\subset |S|$ denote the set of all closed points $ s\in |S|$ with residue degree $ [k(s):k]\leq d  $.

 \begin{conjecture}\label{MC} For every integer $d\geq 1$, one has 
 $$\widetilde{\Ob}_{\Z_\ell}^{\leq d}:=\hbox{\rm sup}\lbrace \widetilde{\Ob}_{\Z_\ell,s}\;|\; s\in |S|^{\leq d}\rbrace <+\infty $$
 and $\widetilde{\Ob}_{\Z_\ell}^{\leq d}=1$, $\ell\gg 0$.
 \end{conjecture}

\noindent  Our first main result is that Conjecture \ref{MC} holds when $S$ is a curve \textit{modulo}  some reasonable  variational realization conjecture, which we discuss now.

 \begin{itemize}[leftmargin=*,parsep=0cm,itemsep=0.2cm,topsep=0.2cm,label=-]

 \item {\bf Singular cohomology}: Fix an embedding $\infty: k\hookrightarrow \C$, let $(-)_\infty$ denote the base-change functor along $\Spec(\C)\stackrel{\infty}{\rightarrow}\Spec(k)$ and $(-)^{\an}$   the analytification functor from varieties over $\C$ to complex analytic spaces.   For every $  s_\infty\in S_\infty(\C)$   the cycle class maps for singular cohomology $$c: \CH^i(X_\infty)_{\Q}\rightarrow   \SH^{2i} (X_\infty^{\an},\Q(i)),\;\; c_{s_\infty}: \CH^i(X_{s_\infty})_{\Q}\rightarrow \SH^{2i} (X_{s_\infty}^{\an},\Q(i)) $$  fit into a canonical commutative diagram 
$$\xymatrix{\CH^i(X_\infty)_{\Q}\ar[rr]^{|_{X_{\infty,s}}}\ar[d]_{c }&&\CH^i(X_{s_\infty})_{\Q}\ar[d]^{c_{s_\infty}}\\
 \SH^{2i} (X_\infty^{\an},\Q(i))\ar[r]^(.4){\epsilon}&\SH^0(S_\infty^{\an},R^{2i}f_{\infty*}^{\an}\Q(i)) \ar@{^{(}->}[r]&\SH^{2i} (X_{s_\infty}^{\an},\Q(i)),}$$
 where  $\epsilon:  \SH^{2d} (X_{\infty}^{\an},\Q(i))\twoheadrightarrow E_{\infty}^{0,i}\hookrightarrow E_2^{0,i}=\SH^0(S_{\infty}^{\an},R^{2i}f_{\infty*}^{\an}\Q(i)) $
is the   edge morphism   from the Leray spectral sequence  for $f_{\infty}^{\an}:X_{\infty}^{\an}\rightarrow S_{\infty}^{\an}$.  

 \begin{enumerate}[leftmargin=*,parsep=0cm,itemsep=0.2cm,topsep=0.2cm,label=VSing{$^0(f_\infty,i)$}]
 \item For every  $  s_\infty\in S_\infty(\C)$   and $\alpha_{s_\infty}\in \SH^0(S^{\an}_{\infty}, R^{2i}f_{\infty*}^{\an}\Q(i))\subset  \SH^{2i}(X^{\an}_{s_\infty}, \Q(i))$ the following properties are equivalent:
  \begin{enumerate}[leftmargin=*,parsep=0cm,itemsep=0.2cm,topsep=0.2cm,label=\arabic*)]
 \item   $\alpha_{s_\infty}\in \hbox{\rm im}[c_{s,\Q}:\CH^i(X_{s_\infty})_{\Q}\rightarrow \SH^{2i}(X^{\an}_{s_\infty}, \Q(i))]$;  
 \item there exists $\widetilde{\alpha}\in \CH^i(X_{\infty})_{\Q}$ such that $c_{s_\infty}(\widetilde{\alpha}|_{X_{s_\infty }})=\alpha_{s_\infty}$. 
 \end{enumerate}
 \end{enumerate}
 \noindent   Though it does not involve  Hodge classes, the statement VSing$^0(f_\infty,i)$ is often referred to as the variational Hodge conjecture for codimension $i$ cycles because, by the fixed part theorem, it  follows from the Hodge conjecture for any  smooth compactification of $X_\infty$  - see \textit{e.g.} \cite[\S 3.1]{CSchnell} for details and an equivalent formulation using de Rham cohomology.  \textit{A priori} the  statement  $\hbox{\rm VSing}^0(f_\infty,i)$ is  not preserved by base-change along finite  covers of smooth varieties while the obstructions $\widetilde{\Ob}_{\Z_\ell,s}$, $s\in S$ are. So we will rather consider  the following   "stabilized" variant VSing$(f_\infty,i)$. For   finite  covers $S''_\infty\rightarrow S'_\infty\rightarrow S_\infty$ of smooth varieties, consider the notation in the  base-change diagram:
  $$\xymatrix{X''_\infty\ar[r]\ar[d]_{f''_\infty}\ar@{}[dr]|\square&X'_\infty\ar[r]\ar[d]_{f'_\infty}\ar@{}[dr]|\square&X_\infty\ar[d]^{f_\infty}\\
 S''_\infty\ar[r]& S'_\infty\ar[r]& S_\infty.}$$ 
  \begin{enumerate}[leftmargin=*,parsep=0cm,itemsep=0.2cm,topsep=0.2cm,label=VSing{$(f_\infty,i)$}]
 \item There exists a    finite  cover $S'_\infty\rightarrow S_\infty$ of   smooth varieties over $\C$ such that for every finite  cover $S''_\infty\rightarrow S'_\infty $ of smooth varieties over $\C$,  VSing$^0(f''_\infty,i)$ holds.
 \end{enumerate}

  \item  {\bf \'Etale $\Q_\ell$-cohomology}: The following is the $\Q_\ell$-\'etale counterpart of VSing$^0(f_\infty,i)$: 
 \begin{enumerate}[leftmargin=*,parsep=0cm,itemsep=0.2cm,topsep=0.2cm,label=VEt{$_{\Q_\ell}^0(f,i)$}]
 \item For every $  s\in |S|$ and $\alpha_s\in \SH^0(S_{\bar k}, R^{2i}f_*\Q_\ell(i))\subset  \SH^{2i}(X_{\bar s}, \Q_\ell(i))$ the following properties are equivalent:
  \begin{enumerate}[leftmargin=*,parsep=0cm,itemsep=0.2cm,topsep=0.2cm,label=\arabic*)]
 \item   $\alpha_s\in \hbox{\rm im}[c_{X_{\bar s},\ell}:\CH^i(X_{\bar s})_{\Q}\rightarrow \SH^{2i}(X_{\bar s}, \Q_\ell(i))]$;  
 \item there exists $\widetilde{\alpha}\in \CH^i(X_{\bar k})_{\Q}$ such that $c_{X_{\bar s}, \ell}(\widetilde{\alpha}|_{X_{\bar s}})=\alpha_s$. 
 \end{enumerate}
 \end{enumerate}
One could also consider the seemingly weaker variant WVEt$_{\Q_\ell}^0(f,i)$ where $\CH^i(X_{\bar s})_{\Q}$, $ \CH^i(X_{\bar k})_{\Q}$ are replaced with  $\CH^i(X_{\bar s})_{\Q_\ell}$, $ \CH^i(X_{\bar k})_{\Q_\ell}$,
 and the stabilized variants $   \hbox{\rm WVEt}_{\Q_\ell}(f,i)$, $  \hbox{\rm VEt}_{\Q_\ell}(f,i)$. Note that the statements $\hbox{\rm WVEt}^0_{\Q_\ell}(f,i)$, $\hbox{\rm VEt}^0_{\Q_\ell}(f,i)$ also make sense  when $p>0$.
\end{itemize}
  
 \begin{proposition}\label{Prop:VConjComp} If $p=0$,  one  has
  $$   \hbox{\rm WVEt}^0_{\Q_\ell}(f,i) \Leftrightarrow  \hbox{\rm VEt}^0_{\Q_\ell}(f,i) \Leftrightarrow \hbox{\rm VSing}^0(f_\infty,i).$$
In general, one always has   $\hbox{\rm VEt}^0_{\Q_\ell}(f,i)\Rightarrow \hbox{\rm WVEt}^0_{\Q_\ell}(f,i)$ and  $\hbox{\rm Tate}_{\Q_\ell}(X_\eta,i)\Rightarrow \hbox{\rm WVEt}_{\Q_\ell}(f,i)$.
\end{proposition}
\noindent We will give a proof of this proposition in section \ref{Sec:ProofPropConjComp}.

 \noindent In particular, when $p=0$,    $\hbox{\rm VSing}^0(f_\infty,i)$ is independent of the embedding $\infty:k\hookrightarrow \C$ and $\hbox{\rm WVEt}^0_{\Q_\ell}(f,i)$,  $\hbox{\rm VEt}^0_{\Q_\ell}(f,i)$ are  independent of the prime $\ell$.  \\

\noindent We can now state our first main result.

\begin{thmABC}\label{MainTh1} Assume $S$ is a curve and  VSing$(f_\infty,i)$ holds  for one (equivalently every) embedding $\infty:k\hookrightarrow \C$.
Then, for every integer $d\geq 1$, one has 
  $\widetilde{\Ob}_{\Z_\ell}^{\leq d}  <+\infty $ and  $\widetilde{\Ob}_{\Z_\ell}^{\leq d}=1$ for $\ell\gg 0$ (depending on $d$).
\end{thmABC}

  \subsubsection{}\hspace{-0.3cm}\label{Sec:MainThm2}Assume now $p>0$. One has a variant of Theorem \ref{MainTh1} for $d=1$  but it is slightly more technical. To state it, one has to make a mild assumption on the $\Q_\ell$-local system $\cV_{\Q_\ell}:=R^{2i}f_*\Q_\ell(i)$, namely that it is GLU   - see Subsection \ref{Sec:GenericPtsSpar} for the definition. One also needs a substitute for  VSing$(f_\infty,i)$.  According to Proposition \ref{Prop:VConjComp},  a first  substitute is WVEt$_{\Q_\ell}(f,i)$. Another natural  substitute  is  the variational realization conjecture in crystalline cohomology  VCrys$(f,i)$.  This is more subtle. Indeed, as  crystalline cohomology is only well-behaved over a perfect residue field,   one has  first  to spread  out all the involved data over a finite base field. Another difficulty is that the proof of Theorem \ref{MainTh1} heavily relies on Artin's comparison isomorphism bewteen \'etale and singular cohomology. But there is no such a direct functorial  comparison isomorphism between crystalline and  \'etale cohomology; to remedy this, one has to invoke   a weak form -  CrysEt$_{\Q_\ell}(f,i)$ of the motivic conjecture predicting that homological and numerical equivalence should coincide (combined with a  theorem of Ambrosi  - see Fact \ref{Emiliano}).  \\

\noindent We now state  VCrys$(f,i)$ and  CrysEt$_{\Q_\ell}(f,i)$. Let $F$ denote the algebraic closure of $\F_p$ in $k$ and let  $\mathscr{K}$ be a smooth, separated, geometrically connected scheme over $F$ with generic point $\eta_{\mathscr{K}}: \Spec(k)\rightarrow \mathscr{K}$, let $\mathcal{S}\rightarrow \mathscr{K}$ be a smooth, separated and geometrically connected morphism and  $f:\mathcal{X}\rightarrow \mathcal{S}$ a smooth proper morphism  fitting in the  following  Cartesian diagram
$$\xymatrix{\mathcal{X}\ar[r]^f\ar@{}[dr]|\square& \mathcal{S}\ar[r]\ar@{}[dr]|\square&\mathscr{K}\\
X\ar[r]_f\ar[u]& S\ar[r]\ar[u]&k\ar[u]_{\eta_{\mathscr{K}}}}$$
\noindent Let $K$ denote the fraction field of the ring $W$ of Witt vectors of $F$. For a $F$-scheme $\cZ$,   write   $ \SH^i_{\crys}(\mathcal{Z}):= \SH^i_{\crys}(\mathcal{Z}/W)_K$ for  the crystalline cohomology with $K$-coefficients and  
 $$c_{\crys}: \CH^i(\mathcal{Z})_{\Q}\rightarrow   \SH^{2i}_{\crys}(\mathcal{Z})  $$ for  the cycle class map. For every $t\in |\mathcal{S}|$  the cycle class maps $$c_{\crys}: \CH^i(\mathcal{X}) \rightarrow   \SH^{2i}_{\crys}(\mathcal{X}),\;\; c_{\crys,t}: \CH^i(\mathcal{X}_t) \rightarrow \SH^{2i}_{\crys}(\mathcal{X}_t)  $$  fit into a canonical commutative diagram 
$$\xymatrix{\CH^i(\mathcal{X})_{\Q}\ar[rr]^{|_{\cX_{t}}}\ar[d]_{c_{\crys}}&&\CH^i(\mathcal{X}_t)_{\Q}\ar[d]^{c_{\crys,t}}\\
\SH^{2i}_{\crys}(\mathcal{X})\ar[r]^(.3){\epsilon}&\SH^0(\mathcal{S},R^{2i}f_{\crys,*}\mathcal{O}_{\mathcal{X}/W})_K\ar@{^{(}->}[r]&\SH^{2i}_{\crys}(\mathcal{X}_t),}$$
where $\epsilon: \SH^{2i}_{\crys}(\mathcal{X})\twoheadrightarrow E_{\infty}^{0,i}\hookrightarrow\SH^0(\mathcal{S},R^{2i}f_{\crys,*}\mathcal{O}_{\mathcal{X}/W})_K$
is, again, the  edge morphism from the Leray spectral sequence  for $f: \cX\rightarrow \cS$ 
in crystalline cohomology - see \cite[\S 1]{Morrow} and the references therein for details.  The following   is the crystalline analogue of  $\hbox{\rm VSing}^0(f_\infty,i)$, $\hbox{\rm VEt}^0_{\Q_\ell}(f,i)$ \cite[Conj. 0.1]{Morrow}. \\

 \begin{enumerate}[leftmargin=*,parsep=0cm,itemsep=0cm,topsep=0cm,label=VCrys{$^0(f,i)$}]
 \item For every  $t\in |\mathcal{S}|$ and $\alpha_t\in \SH^0(\mathcal{S},R^{2i}f_{\crys,*}\mathcal{O}_{\mathcal{X}/W})_{\Q}\subset \SH^{2i}_{\crys}(\mathcal{X}_t)$ the following properties are equivalent.
  \begin{enumerate}[leftmargin=*,parsep=0cm,itemsep=0.2cm,topsep=0.2cm,label=\arabic*)]
 \item   $\alpha_t\in   \hbox{\rm im}[ c_{\crys,t}: \CH^i(\mathcal{X}_t)_{\Q}\rightarrow \SH^{2i}_{\crys}(\mathcal{X}_t)]$; 
  \item there exists $\widetilde{\alpha}\in \CH^i(\mathcal{X})_{\Q}$ such that $c_{\crys,t}(\widetilde{\alpha}|_{\mathcal{X}_t})=\alpha_t$.
 \end{enumerate}
 \end{enumerate}
 \noindent As before, let VCrys{$(f,i)$ denote its stabilized variant.\\
 
  \noindent Also, consider the  following statement 
    \begin{enumerate}[leftmargin=*,parsep=0cm,itemsep=0.2cm,topsep=0.2cm,label=CrysEt{$_{\Q_\ell}(f,i)$}]
 \item  For every $t\in |\mathcal{S}|$, the kernel of the cycle class maps
$$  c_{\crys,t}: \CH^i(\mathcal{X}_t)_{\Q}\rightarrow \SH^{2i}_{\crys}(\mathcal{X}_t), \;\;  c_{\ell,t}: \CH^i(\mathcal{X}_t)_{\Q}\rightarrow \SH^{2i} (\mathcal{X}_{\bar{t}},\Q_\ell) $$ 
coincide,
\end{enumerate}
which  follows from  the standard conjecture predicting that   homological   and numerical equivalences should coincide, which, in turn, is a   consequence of the conjecture predicting that the category of effective motives should be abelian semisimple \cite{Jannsen}.\\

 \noindent We can now state the analogue of Theorem \ref{MainTh1} when $p>0$.

\begin{thmABC}\label{MainTh2} Assume $S$ is a curve,  $\cV_{\Q_\ell}$ is GLU and either 
(i)  $\hbox{\rm WVEt}_{\Q_\ell}(f,i)$ or (ii) 
 $\hbox{\rm VCrys}(f,i) + \hbox{\rm CrysEt}_{\Q_\ell}(f,i)$ holds.
Then,  one has $\widetilde{\Ob}_{\Z_\ell}^{\leq 1}  <+\infty  $.
\end{thmABC}
\begin{remarque}We do not know if, under the assumptions of Theorem \ref{MainTh2},  $ \widetilde{\Ob}_{\Z_\ell}^{\leq 1}=0$, $\ell\gg 0$.  
\end{remarque}

 \subsubsection{Unramified cohomology}\label{Sec:UCoh}   When $i=2$, $(\widetilde{C}_{\Z_\ell,s})_{\tors}$ can be  described in terms of degree $3$ unramified cohomology. More precisely, set $C_{\Z_\ell,s}:=V_{\Z_\ell}/V_{\Z_\ell,s}^a$. From the short exact sequence 
 $$0\rightarrow  \widetilde{C}_{\Z_\ell,s}\rightarrow C_{\Z_\ell,s} \rightarrow V_{\Z_\ell,s}/\widetilde{V}_{\Z_\ell,s}\rightarrow 0$$
and the fact that $ V_{\Z_\ell,s}/\widetilde{V}_{\Z_\ell,s}$ is torsion-free, one has $(\widetilde{C}_{\Z_\ell,s})_{\tors}=(C_{\Z_\ell,s})_{\tors}$. If $i=2$, \cite[Thm. 2.2]{CTK} states that $(C_{\Z_\ell,s})_{\tors}$ is isomorphic to 
$$ \SH^3_{nr}(X_{\bar s},\Q_\ell/\Z_\ell(2))_{\hbox{\rm \tiny ndiv}}\stackrel{def}{=}\hbox{\rm coker}[\SH^3_{nr}(X_{\bar s},\Q_\ell/\Z_\ell(2))_{\hbox{\rm \tiny div}}\rightarrow \SH^3_{nr}(X_{\bar s},\Q_\ell/\Z_\ell(2))].$$
Here for an abelian  group $A$, we  let $A_{\hbox{\rm \tiny div}}\subset A$ denote its maximal divisible subgroup.\\

\noindent  Hence  Theorem \ref{MainTh1} and Theorem \ref{MainTh2} imply:
\begin{corollaire}\label{Cor:H3nr}Assume $S$ is a curve.
\begin{enumerate}[leftmargin=*,parsep=0cm,itemsep=0.2cm,topsep=0.2cm ]
\item Assume $p=0$ and   VSing$(f_\infty,i)$ for some embedding $\infty:k\hookrightarrow \C$ holds. Then, for every integer $d\geq 1$, $$\hbox{\rm sup}\lbrace |\SH^3_{nr}(X_{\bar s},\Q_{\ell} /\Z_{\ell} (2))_{\hbox{\rm \tiny ndiv}}|\; |\;  s\in |S|^{\leq d}  \rbrace | <+\infty,$$
and $\SH^3_{nr}(X_{\bar s},\Q_{\ell} /\Z_{\ell} (2))_{\hbox{\rm \tiny ndiv}}=0$, $s\in |S|^{\leq d}$ for $\ell\gg 0$ (depending on  $d$).  
 
\item  Assume $p>0$, $\cV_{\Q_\ell}$ is GLU and either  (i) $\hbox{\rm WVEt}_{\Q_\ell}(f,i)$ or  (ii) VCrys$(f,i)$ + CrysEt$_{\Q_\ell}(f,i)$ holds. 
Then, $$\hbox{\rm sup}\lbrace |\SH^3_{nr}(X_{\bar s},\Q_{\ell} /\Z_{\ell} (2))_{\hbox{\rm \tiny ndiv}}|\; |\;  s\in  S(k) \rbrace | <+\infty,$$  
 and $\SH^3_{nr}(X_{\bar s},\Q_{\ell} /\Z_{\ell} (2))_{\hbox{\rm \tiny ndiv}}=0$, $s\in  S(k)$ for $\ell\gg 0$. 

\end{enumerate}
 
 \end{corollaire}

\subsection{Acknowledgements} 
The second author is partially supported by the NSF DMS-2201195 grant. 

 $$***$$

\noindent In Section \ref{Sec:Et} we review basic properties of cycle class maps for \'etale $\Z_\ell$-cohomology in families, introduce the notion of $\cV_{\Q_\ell}$-generic points and describe the general strategy for the proof of Theorem \ref{MainTh1} and Theorem \ref{MainTh2}. In  Section \ref{Sec:Comp}, we inject comparison with singular cohomology - Subsection \ref{Sec:Sing}, to prove  Proposition \ref{Prop:VConjComp} and conclude the proofs of  Theorem \ref{MainTh1}, and with crystalline cohomology  - Subsection \ref{Sec:Crys}, to  conclude the proof of   Theorem \ref{MainTh2}. In Subsection \ref{Sec:ObstructionIHC}, we also explain how to derive from Theorem \ref{MainTh1} its variant in the setting of the integral Hodge conjecture.

  \section{\'Etale cycle class maps in families and global  strategy}

 \subsection{\'Etale $\Z_\ell$-local systems}\label{Sec:Et}  Let $S$ be a smooth, geometrically connected variety over $k$.  For every $s\in S $,  fix a geometric point $\bar s$ over it and an \'etale path $\alpha_{\bar s}:(-)_{\bar s}\tilde{\rightarrow} (-)_{\bar \eta}$. In particular, for every $\Z_\ell$-local system   $\cV_{\Z_\ell}$  on $S$, one identifies $\cV_{\Z_\ell,\bar s}\tilde{\rightarrow} \cV_{\Z_\ell,\bar \eta}$ equivariantly with respect to the isomorphism of \'etale fundamental groups $\pi_1(S,\bar s)\tilde{\rightarrow} \pi_1(S,\bar \eta)$, $\gamma\mapsto \alpha_{\bar s}\gamma \alpha_{\bar s}^{-1}$. As a result, we will in general omit fiber functors from our notation and simply write
 $$V_{\Z_\ell}:=\cV_{\Z_\ell,\bar s}\tilde{\rightarrow} \cV_{\Z_\ell,\bar \eta}, \;\; V_{\Q_\ell}:=V_{\Z_\ell}\otimes_{\Z_\ell}\Q_\ell.$$
  Let $f :X \rightarrow S $ be a smooth projective morphism.

  \subsubsection{Notational conventions}\label{Sec:NotEt} Consider the $\Z_\ell$-\'etale local system   $\cV_{\Z_\ell}:=R^{2i}f_*\Z_\ell(i)$ on $S$. Let $G_\ell\subset \SGL(V_{\Q_\ell})$ denote the  Zariski-closure  of the image of $\pi_1(S )$ acting on $ V_{\Q_\ell}$; let also  $\overline{G}_\ell\subset  G_\ell$ and, for every $s\in S$,    $ G_{\ell,s}\subset G_\ell$ denote the Zariski closure of the images of $\pi_1 (S_{\bar k} )$ and $\pi_1(s )$ acting on $V_{\Q_\ell}$ by restriction along the functorial morphisms  $\pi_1(S_{\bar k} )\rightarrow \pi_1(S )$ and $\pi_1(s )\rightarrow \pi_1(S )$ respectively (in particular $G_{\ell,\eta}=G_\ell$).  As $S$ is geometrically connected over $k$,   the functorial sequence   $$1\rightarrow \pi_1(S_{\bar k}  )\rightarrow  \pi_1(S )\rightarrow \pi_1(k)\rightarrow 1$$ is exact, hence  $\overline{G}_\ell\subset G_\ell$ is a normal subgroup, and  for every closed point $s\in |S|$, one has $G_\ell^\circ=\overline{G}_\ell^\circ G_{\ell,s}^\circ$.

  \subsubsection{Specialization and extension of algebraically closed fields}\label{Sec:SpecializationEt} We recall the following two properties of  the cycle class map for \'etale $\Z_\ell$-cohomology. \\

 \paragraph{\textit{Compatibility with specialization  of  algebraic cycles}}\label{Sec:SpecializationEt2}  For every $s\in S$, one has a commutative diagram 
  $$\xymatrix{  \CH^i(X_{\bar k})\ar[r]^{|_{X_{\bar \eta}}}\ar[d]_{|_{X_{\bar s}}}&  \CH^i(X_{\bar\eta})\ar[d]^{c_{\ell,\eta}}\ar[dl]^{sp_{\eta,s}}  \\
  \CH^i(X_{\bar s})\ar[r]_{c_{\ell,s}}& V_{\Z_\ell}\\}$$
(see \cite[$\S$ 20.3, Ex. 20.3.1 and 20.3.5]{Fulton}).\\

 \paragraph{\textit{"Invariance" under extension of algebraically closed field}}\label{Sec:SpecializationEt1} Let  $\Omega  \hookrightarrow \Omega'$ be an extension of algebraically closed fields of characteristic $\not= \ell$ and let $Y$ be a smooth  proper  variety over $\Omega$. Consider the canonical commutative square
 $$\xymatrix{\CH^i(Y)\ar[d]_{|_{Y_{\Omega'}}}\ar[r]^{c_\ell}& \SH^{2i}(Y,\Z_\ell(i))\ar[d]^\simeq\\
 \CH^i(Y_{\Omega'})\ar[r]^{c_\ell}& \SH^{2i}(Y_{\Omega'},\Z_\ell(i)).}$$
 Then\footnote{In fact, a cycle $\xi\in \CH^i(Y_{\Omega'})$ is defined over a finitely generated algebraically closed field $\Omega''\subset  \Omega'$. One could then find a smooth and proper model of $Y$ over a small affine scheme $U$ over $\Omega$ with generic point $\Omega''$ and use the specialization at a $\Omega$-point of $U$, as in \ref{Sec:SpecializationEt2}.},
   $$\Sim[c_\ell\circ -|_{Y_{\Omega'}}]: \CH^i(Y) \rightarrow  \SH^{2i}(Y_{\Omega'},\Z_\ell(i))=\Sim[c_\ell: \CH^i(Y_{\Omega'}) \to  \SH^{2i}(Y_{\Omega'},\Z_\ell(i))].$$
 In particular, $V_{\Z_\ell,s}^{a}$, $V_{\Z_\ell,s}^{\free, a}$ \textit{etc.} are independent of the geometric point $\bar s$ over $s$.
\textit{}\\

 \subsubsection{The lattice $\Lambda_{\Z_\ell}$}\label{Sec:LatticeEt}For every $  s \in S$, define 
 $$\Lambda_{\Z_\ell, s}:=\Sim[\CH^i(X_{\bar k})_{\Z_\ell} \rightarrow \CH^i(X_{\bar s})_{\Z_\ell}\stackrel{c_{\ell,s}}{\rightarrow} V_{\Z_\ell}^{\free}]\subset  V_{\Z_\ell}^{\free}.$$
 \noindent By construction and \ref{Sec:SpecializationEt}, one has  $$\Lambda_{\Z_\ell, s}\subset V^{\free,a}_{\Z_\ell,\eta}\subset V^{\free,a}_{\Z_\ell,s}\subset V^{\free}_{\Z_\ell}.$$
  \begin{lemme}\label{Lem:LatticeEt}   The lattice $\Lambda_{\Z_\ell}:=\Lambda_{\Z_\ell, s}\subset V_{\Z_\ell}^{\free}$ is  independent of $s $  (\textit{modulo} the identifications $V_{\Z_\ell}=\cV_{\Z_\ell,\bar s }\simeq \cV_{\Z_\ell,\bar \eta}$).  \end{lemme}
\begin{proof} This follows from the fact that the restriction morphism $\SH^{2i}(X_{\bar k},\Z_\ell(i))\rightarrow \SH^{2i}(X_{\bar s},\Z_\ell(i))=V_{\Z_\ell}$
factors through the edge morphism  $\epsilon: \SH^{2i}(X_{\bar k},\Z_\ell(i))\twoheadrightarrow E_{\infty}^{0,i}\hookrightarrow E_2^{0,i}=\SH^0(S_\infty,R^{2i}f_*\Z_\ell(i))$ of  the Leray spectral sequence  for $f:X\rightarrow S$ as 
$$\xymatrix{\CH^i(X_{\bar k})_{\Z_\ell}\ar[rr]^{|_{X_{\bar s}}}\ar[d]_{c_{\ell }} &&\CH^i(X_{\bar s})_{\Z_\ell}\ar[d]^{c_{\ell,s}}\\
\SH^{2i}(X_{\bar k},\Z_\ell(i))\ar[r]^(.4){\epsilon}&\SH^0(S_{\bar k},R^{2i}f_*\Z_\ell(i))\ar[r]^(.6){(-)_{\bar s}}&V_{\Z_\ell}^{\free}}$$
and the fact the embedding  
 $$ V_{\Z_\ell}^{\free}\cap (V_{\Q_\ell})^{\overline{G}_\ell}=\Sim[\SH^0(S_{\bar k},R^{2i}f_*\Z_\ell(i))\stackrel{(-)_{ \bar s }}{\rightarrow}  V^{\free}_{\Z_\ell}]\subset V_{\Z_\ell}^{\free}$$
 is  independent of $s $  (\textit{modulo} the identifications $V_{\Z_\ell}=\cV_{\Z_\ell,\bar s }\simeq \cV_{\Z_\ell,\bar \eta}$).  
     \end{proof}
\begin{remarque}\label{Rem:LatticeEt}\textnormal{Assume\footnote{If $p=0$, this is always the case - see \cite{Nagata1}, \cite{Nagata2}, \cite{Hironaka}.} there exists a smooth compactification $X \hookrightarrow X ^{\cpt}$. Then the surjectivity of  the restriction morphism  $\CH^i(X_{\bar k}^{\cpt}) \twoheadrightarrow \CH^i(X_{\bar k})$ and the functoriality of cycle class maps 
shows that $\Lambda_{\Z_\ell}$ can also be described as
$$\Lambda_{\Z_\ell}=\Sim[\CH^i(X_{\bar k}^{\cpt})_{\Z_\ell}\stackrel{c_\ell}{\rightarrow} \SH^{2i}(X_{\bar k}^{\cpt},\Z_\ell(i))\rightarrow\SH^{2i}(X_{\bar s}^{\cpt},\Z_\ell(i))\twoheadrightarrow V^{\free}_{\Z_\ell}].$$
In particular, if $\bar k\hookrightarrow\Omega$ is an extension of algebraically closed fields and $s_\Omega$ a geometric point on $S_\Omega$ over $\bar s$, then \ref{Sec:SpecializationEt1}  shows that 
$$\Lambda_{\Z_\ell}=\Sim[\CH^i(X_{\Omega})_{\Z_\ell} \rightarrow \CH^i(X_{s_\Omega})_{\Z_\ell}\stackrel{c_{\ell,s_\Omega}}{\rightarrow} V_{\Z_\ell}^{\free}].$$}
\end{remarque}

 \subsection{Strategy for the proof of Theorem \ref{MainTh1} and Theorem \ref{MainTh2}}  We retain the notation and conventions of Subsection \ref{Sec:Statements} and Subsection \ref{Sec:NotEt}. For every $s\in S$, set 
  $$\Ob_{\Z_\ell,s}^{\free}:=|(C^{\free}_{\Z_\ell,s})_{\tors}|.$$
As
 $$\widetilde{\Ob}_{\Z_\ell,s}\leq |(V_{\Z_\ell})_{\tors}| \Ob^{\free}_{\Z_\ell,s}$$
and  as $(V_{\Z_\ell})_{\tors}$ is independent of $s\in S$ and, if\footnote{This follows from Artin's comparison - see Subsection \ref{Sec:ArtinComp} and the fact that singular cohomology groups are finitely generated. This is also true if $p>0$ \cite{GabberTorsion} but we will not resort to this fact.} $p=0$,  $(V_{\Z_\ell})_{\tors}=0$, $\ell\gg 0$ , it is enough to prove Theorem \ref{MainTh1}, Theorem \ref{MainTh2} for $ \Ob_{\Z_\ell,s}^{\free}$ instead of $\widetilde{\Ob}_{\Z_\ell,s}$.

  \subsubsection{$\cV_{\Q_\ell}$-generic points}\label{Sec:GenericPts}    The proofs of Theorem \ref{MainTh1} and Theorem \ref{MainTh2} are parallel and follow from the combination of two independent statements involving $\cV_{\Q_\ell}$-generic points.  Let   $\cV_{\Z_\ell}$ be a  $\Z_\ell$-local system on $S$. \\

  \paragraph{\textit{$\cV_{\Q_\ell}$-generic points}}  Define  the sets of   closed $\cV_{\Q_\ell}$-generic points  to be the subset $|S|_{\cV_{\Q_\ell}}^\gen \subset |S|$   of all $s\in |S|$ satisfying the following equivalent conditions $$
  G_{\ell,s}^\circ = G_\ell^\circ \Leftrightarrow 
  G_{\ell,s}^\circ \supset G_\ell^\circ  \Leftrightarrow 
 G_{\ell,s}^\circ \supset \overline{G}_\ell^\circ,$$
\noindent  and let $  |S|_{\cV_{\Q_\ell}}^{\ngen}:= |S|\setminus |S|_{\cV_{\Q_\ell}}^\gen\subset |S|$
be  the subset  of closed  non-$\cV_{\Q_\ell}$-generic  points.  Note that    $|S|_{\cV_{\Q_\ell,}}^\gen $ is contained in the set of all $s\in |S|$ such that $V_{\Q_\ell,s}^a\subset (V_{\Q_\ell })^{\overline{G}_\ell^\circ}$.\\

  \paragraph{\textit{Sparcity}}\label{Sec:GenericPtsSpar}  Under mild assumptions one expects non-$\cV_{\Q_\ell}$-generic points to be sparce - see  \cite{Bas} for details. 
  When $S$ is a curve, one has the following unconditional results. Let $\overline{\Pi}_\ell$ denote the image of $\pi_1(S_{\bar k})$ acting on $V_{\Q_\ell}$ and, if $p>0$, let $\overline{\Pi}^+_\ell(\supset \overline{\Pi}_\ell)$ denote the image of  $\pi_1(S_{k\bar \F_p})$ acting on $V_{\Q_\ell}$; these are $\ell$-adic Lie groups. One says that $\cV_{\Q_\ell}$ is:
  \begin{itemize}[leftmargin=*,parsep=0cm,itemsep=0.2cm,topsep=0.2cm,label=-]
  \item     GLP (geometrically Lie perfect) if Lie$(\overline{\Pi}_\ell)$ is a perfect Lie algebra \textit{viz} one has  $[\hbox{\rm Lie}(\overline{\Pi}_\ell),\hbox{\rm Lie}(\overline{\Pi}_\ell)]=0$;  
  \item and, if $p>0$, GLU (geometrically Lie unrelated) if $\hbox{\rm Lie}(\overline{\Pi}_\ell)$ and $\hbox{\rm Lie}(\overline{\Pi}_\ell^+)$ have no non-trivial common quotient.
  \end{itemize}

 \begin{factABC}\label{GenericPts:Spar1}\textit{}\textnormal{(\cite[Thm. 1]{UOI2}).} Assume $p=0$, $S$ is a curve and $\cV_{\Q_\ell}$ is GLP. Then for every integer $d\geq 1$,  the set $|S|_{\cV_{\Q_\ell}}^\ngen\cap |S|^{\leq d}$ is finite.
 \end{factABC}

  \begin{factABC}\label{GenericPts:Spar2}\textit{}\textnormal{(\cite{TamagawaUOI}; see also   the discussion in  \cite[1.7.1]{Emiliano}).} Assume $p>0$, $S$ is a curve and $\cV_{\Q_\ell}$ is GLU. Then    the set $|S|_{\cV_{\Q_\ell}}^\ngen\cap  S(k)$ is finite.
 \end{factABC}

 \noindent The $\Z_\ell$-local system $\cV_\ell=R^{2i}f_*\Q_\ell(i)$ is GLP  \cite{Hodge2}, \cite{Weil2}. If $p>0$, it is not necessarily GLU  but still, it is \textit{e.g.} if $\overline{\Pi}_\ell$ is open in the derived subgroup of  the image of $\pi_1(S_{\bar k})$ acting on $V_{\Q_\ell}$   - see \cite[Rem. 1.7.1.4]{Emiliano} for details.

 \subsubsection{The main Lemmas}Fact \ref{GenericPts:Spar1}  immediately reduce  the proof   of Theorem \ref{MainTh1}   to the proof of: 
   
 \begin{lemABC}\label{MainLem1} Set $\cV_{\Z_\ell}:=R^{2i}f_*\Z_\ell(i)$. Assume $p=0$  and   $\hbox{\rm VSing}(f_\infty,i)$  holds for some (equivalently every) embedding $\infty:k\hookrightarrow \C$. Then,
  $$\Ob_{\Z_\ell}^{\free,\gen}:=\hbox{\rm sup}\lbrace \Ob^{\free}_{\Z_\ell,s}\;|\; s\in |S|^{\gen}_{\cV_{\Q_\ell}}\rbrace <+\infty, $$
and   $\Ob_{\Z_\ell}^{\free, \gen}=1$ for $\ell\gg 0$.
 \end{lemABC}
 \noindent The proof of Lemma \ref{MainLem1} will be carried out  in Section \ref{Sec:MainLem1Proof}.\\

  \noindent Similarly, Fact \ref{GenericPts:Spar2}  immediately reduces the proof   of Theorem \ref{MainTh2}   to the proof of: 
 
  \begin{lemABC}\label{MainLem2}  Set $\cV_{\Z_\ell}:=R^{2i}f_*\Z_\ell(i)$. Assume $p>0$ and either   (i) $\hbox{\rm WVEt}_{\Q_\ell}(f,i)$  or 
   (ii) $\hbox{\rm VCrys}(f,i)+ \hbox{\rm CrysEt}_{\Q_\ell}(f,i)$ holds. Then,  $\Ob_{\Z_\ell}^{\free,\gen}  <+\infty$.
 \end{lemABC}
 
  \noindent The proof of Lemma  Lemma \ref{MainLem2} will be carried out in Section \ref{Sec:LemmaBiiproof}.\\
 
 \noindent Note that Lemma \ref{MainLem1} and Lemma \ref{MainLem2} do not involve any restriction on the dimension of $S$ nor on the degree of the residue field $k(s)$ for $s\in  |S|^{\gen}_{\cV_{\Q_\ell}}$.

  \begin{remarque}\textnormal{ \textit{A priori}, the assumptions in   Lemma \ref{MainLem1}, Lemma \ref{MainLem2}  do not imply  $\hbox{\rm Tate}_{\Q_\ell}(X_s,i)$, $s\in |S|^{\gen}_{\cV_{\Q_\ell}}$.   However, if one assumes Tate$ _{\Q_\ell}(X_{s_0},i)$ holds for some $s_0\in |S|^{\gen}_{\cV_{\Q_\ell}}$ then these assumptions   indeed  imply  $\hbox{\rm Tate}_{\Q_\ell}(X_s,i)$, $s\in  |S|^{\gen}_{\cV_{\Q_\ell}}$. Indeed,   the proofs of  Lemma \ref{MainLem1}, Lemma \ref{MainLem2} will show these assumptions imply  $\Lambda_{\Q_\ell}= V_{\Q_\ell,s}^a$, $s\in  |S|^{\gen}_{\cV_{\Q_\ell}}$, where  $\Lambda_{\Q_\ell}=\Lambda_{\Z_\ell}\otimes_{\Z_\ell}\Q_\ell$. Assume furthermore Tate$ _{\Q_\ell}(X_{s_0},i)$ holds - that is $V_{\Q_\ell,s_0}^a=\widetilde{V}_{\Q_\ell,s_0}$, for some $s_0\in |S|^{\gen}_{\cV_{\Q_\ell}}$. But then,    for every $s\in  |S|^{\gen}_{\cV_{\Q_\ell}}$, one has 
  $$V_{\Q_\ell,s}^a=\Lambda_{\Q_\ell}=V_{\Q_\ell,s_0}^a= \widetilde{V}_{\Q_\ell,s_0} \stackrel{(\alpha)}{=}\widetilde{V}_{\Q_\ell,s},$$
  where $(\alpha)$ follows from   $s_0\in|S|^{\gen}_{\cV_{\Q_\ell}}$}.\end{remarque}

 \subsubsection{Reduction to connected  monodromy groups}\label{Sec:ConnectedReduction} To  bound $\Ob^{\free}_{\Z_\ell,s} $  uniformly for $s\in   |S|_{\cV_{\Q_\ell}}^\gen$, one can  freely replace $f:X\rightarrow S$ by a base change along a finite   cover $\pi: S'\rightarrow S$ of connected smooth varieties over $k$. Indeed, consider the base-change diagram 
$$\xymatrix{X'\ar[r]\ar@{}[dr]|{\square}\ar[d]_{f'}&X \ar[d]^f\\
S'\ar[r]&S}$$ 
and write  $ \cV'_{\Z_\ell}:=R^{2i}f'_*\Z_\ell(i)$.
For $s\in |S|$ and  $s'\in |S'| $ over $s\in |S|$, let $\bar s'$ be a geometric point over $s'$ and let $\bar s=\pi\circ \bar s'$ denote its image on $S$.    
Then,   $X'_{\bar{s}'}\tilde{\rightarrow} X_{\bar{s}}$ as $\bar{k}$-schemes hence, \textit{a fortiori}, $\CH^i(X'_{\bar s'})\tilde{\rightarrow} \CH^i(X_{\bar{s}}) $. On the other hand,  by proper base change, $\cV'_{\Z_\ell}=\pi^*\cV_{\Z_\ell}$ hence, one gets a canonical  commutative square
$$\xymatrix{\CH^i(X_{\bar s}) \ar[r]^{c_{ \ell,s}}&\SH^{2i}(X_{\bar s},\Z_\ell(i))\ar@{=}[d]\\
\CH^i(X'_{\bar s'})\ar[r]^{c_{\ell,s'}}\ar[u]^\simeq&\SH^{2i}(X'_{\bar s'},\Z_\ell(i))},$$
where the vertical arrows are isomorphisms and the right vertical one is equivariant with respect to the functorial morphism $\pi_1(S')\hookrightarrow \pi_1(S)$. In particular, as $\pi_1(S')\hookrightarrow \pi_1(S)$ is open, one has    $s\in |S|^{\gen}_{\cV_{\Q_\ell}}$ if and only if $s'\in |S'|^{\gen}_{\cV'_{\Q_\ell}}$. \\

\noindent After base change along a finite cover $S'\rightarrow S$ of smooth varieties (which, working componentwise, we may assume to be connected and, replacing $k$ by a finite field extension, geometrically connected over $k$), one may assume VSing$^0(f'_\infty,i)$ (resp. WVEt$^0_{\Q_\ell}(f',i)$, resp.  VCrys$^0(f',i)$) holds for every  base change along a finite cover $S_\infty'\rightarrow S_\infty$ (resp. $S'\rightarrow S$, resp. $\mathcal{S}'\rightarrow \mathcal{S}$) of smooth varieties. Then, 
 the assumptions and conclusions of Theorem \ref{MainTh1} and Theorem \ref{MainTh2} become unchanged by base change along  finite   covers   of smooth varieties, so that  one may  assume:

  \begin{enumerate}[leftmargin=*,parsep=0cm,itemsep=0.2cm,topsep=0.2cm,label=\alph*)]
  
  \item   the algebraic group  $\overline{G}_\ell$   is connected\footnote{\label{ConnectedGeo} First, after replacing $k$ by a finite field extension, one may assume $S(k)\not=\emptyset$, so that fixing $s\in S(k)$ yields a splitting $s:\pi_1(s)=\pi_1(k)\hookrightarrow \pi_1(S)$ of the canonical short exact sequence 
$$1\rightarrow \pi_1(S_{\bar k})\rightarrow \pi_1(S)\rightarrow \pi_1(k)\rightarrow 1$$
and a well-defined action by conjugacy of $\pi_1(k)$ on $\pi_1(S)$.   Then, let  $S'_{\bar k}\rightarrow S_{\bar k}$ denote the connected \'etale cover corresponding to $\ker(\pi_1(S_{\bar k})\rightarrow \pi_0(\overline{G}_\ell))$. As $\overline{G}_\ell^\circ$ is normal in $G_\ell$, the $\pi_1(k)$-action  stabilizes  $\pi_1(S'_{\bar k})$ hence  $s(\pi_1(k))\pi_1(S'_{\bar k})\subset \pi_1(S)$ is an open subgroup corresponding to a connected \'etale cover $S'\rightarrow S$ which, by construction, has the requested property. 
}; 
  \item  the algebraic groups  $G_{\ell,s}$, $s\in S$ are all connected\footnote{After base-change along  the connected \'etale cover     $S'\rightarrow S$ trivializing $\cV_\ell/\tilde{\ell}$ (with $\tilde{\ell}=4$ if $\ell=2$ and $\tilde{\ell}=\ell$ if $\ell\not=2$, this classically follows from the Cebotarev density theorem, using Frobenius tori.}.
  \end{enumerate}

  \subsubsection{An elementary lemma}\label{Sec:Specialization} Recall that for  every $s\in S$,   we identify
 $V_{\Z_\ell}:=\cV_{\Z_\ell,\bar s}\tilde{\rightarrow} \cV_{\Z_\ell,\bar \eta}$.
  For a subset $\Sigma\subset S$, set 
  $$V_{\Z_\ell,\Sigma}^{\free,a}:=\bigcap_{s\in \Sigma} V_{\Z_\ell,s}^{\free,a}\subset V_{\Z_\ell,s}^{\free,a}\subset V_{\Z_\ell}^{\free}.$$

 \begin{lemme}\label{Lem:Spec}  For every $\Z_\ell$-submodule $T_{\Z_\ell}\subset V_{\Z_\ell,\Sigma}^{\free,a}$ and for every $s\in \Sigma$, one has the following implications $$ T_{\Q_\ell}=V_{\Q_\ell,s}^a\Longleftrightarrow [V_{\Z_\ell, s}^{\free,a}:T_{\Z_\ell} ]<+\infty\Longrightarrow \Ob^{\free}_{\Z_\ell,s} \leq  c(T_{\Z_\ell}):=|(V_{\Z_\ell}^{\free}/T_{\Z_\ell})_{\tors}|.$$
 \end{lemme}
 \begin{proof} The first equivalence is straightforward. The second implication follows from the canonical commutative  diagram of short exact sequences 
\begin{equation}\label{Diagram:SpecHodge}\xymatrix{0\ar[r]&T_{\Z_\ell} \ar[r]\ar@{_{(}->}[d]&V_{\Z_\ell}^{\free}\ar[r]\ar@{=}[d]&V_{\Z_\ell}^{\free}/T_{\Z_\ell} \ar[r]\ar@{->>}[d]&0\\
 0\ar[r]&V_{\Z_\ell, s}^{\free,a}\ar[r] &V_{\Z_\ell}^{\free}\ar[r] &C^{\free}_{\Z_\ell,s}\ar[r]&0\\ } \end{equation}
 which, by the snake lemma, identifies 
$$Q_{\Z_\ell,s}:=\hbox{\rm coker}[T_{\Z_\ell} \hookrightarrow V_{\Z_\ell, s}^{\free,a}]\tilde{\rightarrow} \ker[V_{\Z_\ell}^{\free}/T_{\Z_\ell}   \twoheadrightarrow C^{\free}_{\Z_\ell, s}]=:K_{\Z_\ell, s}.$$
But if $K_{\Z_\ell,  s}$ is finite, one gets a short exact sequence 
$$0\rightarrow K_{\Z_\ell, s}\rightarrow  (V_{\Z_\ell}^{\free}/T_{\Z_\ell})_{\tors}\rightarrow (C^{\free}_{\Z_\ell, s})_{\tors}\rightarrow 0,$$
whence the assertion. \end{proof}

\noindent   Lemma \ref{Lem:Spec} reduces the proof of Lemma \ref{MainLem1} and Lemma \ref{MainLem2} to finding a $\Z_\ell$-submodule $T_{\Z_\ell}\subset V_{\Z_\ell,\Sigma}^{\free,a}$ such that $T_{\Q_\ell}=V_{\Q_\ell,s}^a$, $s\in \Sigma= |S|_{\cV_{\Q_\ell}}^{\gen}$ and,  in the setting of of Lemma \ref{MainLem1}, such that $c(T_{\Z_\ell})=0$, $\ell\gg 0$. In all cases, we will consider the $\Z_\ell$-submodule  
$ T_{\Z_\ell} :=\Lambda_{\Z_\ell}$ introduced in Subsection \ref{Sec:LatticeEt}, Lemma \ref{Lem:LatticeEt}.  As a warm-up, we end this Section with the proof of Lemma \ref{MainLem2}   (i).

  \subsubsection{Proof of Lemma \ref{MainLem2}   (i).}\label{Sec:ProofLem2(i)}  Let $s\in \Sigma= |S|_{\cV_{\Q_\ell}}^{\gen}$. Assuming WVEt$_{\Q_\ell}(f,i)$, we are   to prove that the inclusion  $ \Lambda_{\Q_\ell }\subset V_{\Q_\ell ,s}^{ a}$ is an equality.   This follows from the inclusions
 $$ V_{\Q_\ell,s}^a=V_{\Q_\ell,s}^a\cap  \widetilde{V}_{\Q_\ell,s} \stackrel{(\alpha)}{=} V_{\Q_\ell,s}^a\cap  \widetilde{V}_{\Q_\ell,\eta} \stackrel{(\beta)}{\subset} V_{\Q_\ell,s}^a\cap (V_{\Q_\ell})^{\overline{G}_\ell}\stackrel{(\gamma)}{=}  \Lambda_{\Q_\ell} \subset V_{\Q_\ell},$$
   where  $(\alpha)$ follows from $s\in |S|_{\cV_{\Q_\ell}}^{\gen}$,   $(\beta)$ from the reduction   \ref{Sec:ConnectedReduction} a), and   $(\gamma)$ is   WVEt$_{\Q_\ell}(f,i)$.

  \section{Comparison with singular and crystalline cohomologies}\label{Sec:Comp}

  \subsection{Singular cohomology}\label{Sec:Sing}

 \subsubsection{Singular $\Z$-local systems}\label{Sec:SingularBasics} Let $S_\infty$ be a connected variety smooth  over $\C$.   For every $ s_{0\infty}, s_\infty\in S_\infty(\C) =S_\infty^{\an}$, fix a topological path $s_{\infty} \rightarrow   s_{0\infty}$, inducing an isomorphism of fiber functors $\alpha_{s_\infty }: (-)_{s_{\infty} }\tilde{\rightarrow} (-)_{s_{0\infty}}$.  In particular, for every singular $\Z$-local system   $\cV_{ \Z}$  on $S_\infty^{\an}$, one identifies $\cV_{\infty,\Z ,s_{ \infty} }\tilde{\rightarrow} \cV_{\infty,\Z ,s_{0\infty}}$ equivariantly with respect to the isomorphism of  topological fundamental groups $\pi_1^{\topo}(S_\infty^{\an} ,s_{\infty})\tilde{\rightarrow} \pi_1^{\topo}(S_\infty^{\an},s_{0\infty})$, $\gamma\mapsto \alpha_{ s_{\infty} }\gamma \alpha_{s_{\infty} }^{-1}$. So that we will in general omit fiber functors from our notation and simply write
 $$V_{  \Z }:=\cV_{ \Z,s_{\infty}}\tilde{\rightarrow} \cV_{ \Z,s_{0\infty}}.$$
  \noindent Let $f_\infty:X_\infty\rightarrow S_\infty$ be a smooth projective morphism.  The singular $\Z$-local system   $\cV_\Z:=R^{2i}f_{\infty}^{\an}\Z(i)$ on $S_\infty^{\an }$ underlies a polarizable $\Z$-variation of Hodge structure. Let $G\subset \SGL(V_\Q)$ denote the generic Mumford-Tate group of $\cV_\Q:=\cV_\Z\otimes_\Z\Q$, and for every $s_\infty\in S_\infty(\C)$, let $G_{s_\infty}\subset G$ denote the Mumford-Tate group of the polarizable $\Q$-Hodge structure $ s_\infty^*\cV_\Q$. Let also $\overline{G}\subset  \SGL(V_\Q)$ denote the Zariski-closure of the image of $\pi_1^{\topo}(S_\infty^{\an} )$ acting on $V_\Q$. By the fixed part theorem, $\overline{G}^\circ$ a normal closed subgroup of $G$ and,   for every $s_\infty\in S_\infty(\C)$,  one has $G=\overline{G}^\circ G_{s_\infty}$.\\

 \noindent As in Subsection \ref{Sec:LatticeEt}, for every $s_\infty\in S_\infty(\C)$ set 
 $$\Lambda_{\Z,s_\infty}:=\Sim[\CH^i(X_\infty)\rightarrow \CH^i(X_{s_\infty})\stackrel{c_{ s_\infty}}{\rightarrow}V^{\free}_\Z]\subset V^{\free}_\Z.$$
 The same argument as in the proof of Lemma \ref{Lem:LatticeEt} (using Leray spectral sequence for singular cohomology) shows that $\Lambda_{\Z }:=\Lambda_{\Z,s_\infty}$ is independent of $s_\infty\in S_\infty(\C)$.

  \subsubsection{Artin's comparison}\label{Sec:ArtinComp} Assume $p=0$  and fix an embedding $\infty: k\hookrightarrow \C$.  Recall that  $(-)_\infty$ denotes the base-change functor along $\Spec(\C)\stackrel{\infty}{\rightarrow}\Spec(k)$ and $(-)^{\an}$   the analytification functor from varieties over $\C$ to complex analytic spaces.  Let $S$ be a geometrically connected, smooth variety over $k$.   For every $ s_\infty\in S_\infty(\C)$  over $s\in S$ let    $  k(\bar s)\subset \C$ denote the algebraic closure of $k(s)$ determined by $k(s)\hookrightarrow \C$ and let $\bar s$ denote the corresponding geometric point  over $s$.  
 Let $f :X \rightarrow S $ be a smooth projective morphism. The local systems $\cV_\Z:=R^{2i}f_{\infty}^{\an}\Z(i)$ on $S_\infty^{\an}$ and 
  $\cV_{\Z_\ell}:=R^{2i}f_{\infty}^{\an}\Z_\ell(i)$ on $S$ are related by Artin's comparison isomorphism 
   \cite[XI]{SGA4} 
	\begin{equation}\label{Artin}\mathcal{V}_{ \Z}\otimes_\Z\Z_\ell \tilde{\rightarrow} \mathcal{V}_{ \Z_\ell}^{\an},\end{equation}
	where we write $\mathcal{V}_{ \Z_\ell}^{\an}$ for the pull-back of $\mathcal{V}_{ \Z_\ell}$ along\footnote{More precisely, write  $ \mathcal{V}_{\Z_\ell}=\lim_n \mathcal{V}_{ \Z/\ell^n}$ 
	as a limit of $\Z/\ell^n$-local systems  and define the analytification of $ \mathcal{V}_{ \Z_\ell}$ as 
 $(\mathcal{V}_{ \Z_\ell})^\an :=\lim_n \mathcal{V}_{ \Z/\ell^n}|_{(X_\infty^{\an})_{\an}}$.}   the morphisms of sites
 $(X_\infty^{\an})_\an \to X_{\infty,\et} \to X_\et$. Equivalently, for every $s_\infty \in S_\infty(\C) $   over $s\in |S|$, one has a canonical isomorphism of $\Z_\ell$-modules 
 \begin{equation}\label{ArtinFiber} V_\Z\otimes_\Z\Z_\ell =\cV_{ \Z, s_\infty}\otimes_\Z\Z_\ell \tilde{\rightarrow} \cV_{ \Z_\ell,\bar{s}}= V_{ \Z_\ell},\;\; V_\Q\otimes_\Q\Q_\ell\tilde{\rightarrow} V_{\Q_\ell},\end{equation}  which is equivariant with respect to the profinite completion morphism composed with the GAGA isomorphism and the projection
$$\pi_1^\topo(S^{\an}_\infty ) \to \pi_1^\topo(S^{\an}_\infty  )^\wedge  \tilde{\rightarrow}  
\pi_1(S_\infty )  \tilde{\rightarrow} \pi_1(S_{\bar{k}}  )\hookrightarrow \pi_1(S ).$$
In particular,   $\overline{G} \subset \SGL(V_{ \Q})$  identifies, \textit{modulo} (\ref{ArtinFiber}), with the scalar extension $\overline{G}_{\Q_\ell} \subset \SGL(V_{\Q }\otimes_\Q\Q_\ell)$ of $\overline{G} \subset \SGL(V_{\Q })$.  \\

\noindent  Artin's comparison isomorphism  is compatible with cycle class maps on both sides. Namely, for every $ s_\infty\in S_\infty(\C)$  over $s\in S$  one has a canonical  commutative diagram 
 $$\xymatrix{\CH^i(X_{\bar k})\ar[r]^{|_{X_{\bar s}}}\ar[d]_{|_{X_\infty }} &\CH^i(X_{\bar s})\ar[r]^{c_{\ell,s}}\ar[d]_{|_{X_{ s_\infty}}}&V_{\Z_\ell}^{\free}\\ 
 \CH^i(X_{\infty} )\ar[r]_{|_{X_{ s_\infty}}} &\CH^i(X_{s_\infty})\ar[r]_{c_{ s_\infty}}& V_\Z^{\free}\ar@{^{(}->}[r]_{-\otimes_{\Z}\Z_\ell}\ar@{^{(}->}[u]& V_\Z^{\free}\otimes_{\Z}\Z_\ell.\ar[ul]_{(\ref{ArtinFiber})}^\simeq } 
$$
 
\noindent As a result, we will identify subgroups of $ V_\Z^{\free}$ (\textit{e.g.} $\Lambda_\Z$, $V_{\Z,s_\infty}^{\free,a}$ \textit{etc.}) with their image in $V_{\Z_\ell}^{\free}$. Set 
$$\Lambda_{\ell,\Z}:=\Sim[\CH^i(X_{\bar k}) \rightarrow  \CH^i(X_{\bar s}) \stackrel{c_{\ell,s}}{\rightarrow}V_{\Z_\ell}^{\free}]\subset V^{\free,a}_{\ell,\Z,s}:=\Sim[ \CH^i(X_{\bar s}) \stackrel{c_{\ell,s}}{\rightarrow}V_{\Z_\ell}^{\free}].$$
\noindent Then, from \ref{Sec:SpecializationEt1} and Remark \ref{Rem:LatticeEt} applied to $\bar k\hookrightarrow \C$, one has 
 $$\Lambda_\Z=\Lambda_{\ell,\Z},\;\; V^{\free, a}_{\Z,s_\infty}=V^{\free, a}_{\ell,\Z,s},$$
 hence \begin{equation}\label{comp}
 \Lambda_{\ell,\Z}\otimes_\Z\Z_\ell\tilde{\rightarrow}\Lambda_{\Z_\ell},\;\; V^{\free, a}_{\ell,\Z,s}\otimes_\Z\Z_\ell\tilde{\rightarrow} V_{\Z_\ell,s}^{\free, a}.
\end{equation}

 \subsubsection{Proof of Proposition \ref{Prop:VConjComp}}\label{Sec:ProofPropConjComp} 
For every $s\in S$,  write  $$\Lambda_{\ell,\Q} =\hbox{\rm im}[ \CH^i(X_{\bar k})_{\Q } \rightarrow  \CH^i(X_{\bar s})_{\Q }\stackrel{c_{\ell,s}}{\rightarrow}  V_{\Q_\ell} ]\subset V^a_{\ell,\Q,s}:=\hbox{\rm im}[ \CH^i(X_{\bar s})_{\Q }\stackrel{c_{ \ell,s}}{\rightarrow}   V_{\Q_\ell} ]\subset V_{\Q_\ell,s}^a,$$
$$\Lambda_{\Q_\ell}=\hbox{\rm im}[ \CH^i(X_{\bar k})_{\Q_\ell } \rightarrow \CH^i(X_{\bar s})_{\Q_\ell }\stackrel{c_{ \ell,s}}{\rightarrow}  V_{\Q_\ell} ].$$
If $p=0$, fix an embedding  $\infty:k\hookrightarrow \C$ and, for every $s_\infty\in S_\infty(\C)$, write 
 $$\Lambda_\Q = \hbox{\rm im}[\CH^i(X_{\infty})_{\Q } \rightarrow  \CH^i(X_{s_\infty})_{\Q }\stackrel{c_{ s_\infty} }{\rightarrow}  V_{\Q } ]\subset V_{\Q,s_\infty}^a.$$
\noindent  Recall from Subsection \ref{Sec:SingularBasics} and Subsection \ref{Sec:LatticeEt}  that $\Lambda_\Q$ is independent of $s_\infty$ and $\Lambda_{\ell,\Q} $, $\Lambda_{ \Q_\ell} $ are independent of $s$ (as the notation suggests) and, if $p=0$,   from Subsection \ref{Sec:ArtinComp}, that $\Lambda_{\ell,\Q}=\Lambda_\Q$. \\

 \noindent    With these notation, VSing$^0(f_\infty,i)$, VEt$^0_{\Q_\ell}(f,i)$ and  WVEt$^0_{\Q_\ell}(f,i)$ can be reformulated as 
 $$ \begin{tabular}[t]{lll}
VSing$^0(f_\infty,i)$&   $V_{\Q,s_\infty}^a\cap (V_\Q)^{\overline{G}}\subset \Lambda_\Q,$& $  s_\infty\in S_\infty$.  \\
 
VEt$^0_{\Q_\ell}(f,i)$&   $V_{\ell,\Q,s }^a\cap (V_{\Q_\ell})^{\overline{G}_\ell}\subset \Lambda_{\ell,\Q},$& $s\in |S|$.   \\
 
 WVEt$^0_{\Q_\ell}(f,i)$ &   $V_{\Q_\ell,s }^a\cap (V_{\Q_\ell})^{\overline{G}_\ell}\subset  \Lambda_{\Q_\ell},$& $s\in |S|$.\\
 \end{tabular}$$

 \noindent The implication $\hbox{\rm VEt}^0_{\Q_\ell}(f,i)\Rightarrow \hbox{\rm WVEt}^0_{\Q_\ell}(f,i)$ immediately follows from the fact  that, for every $s\in S$,  $V_{\Q_\ell,s}^a$ is the $\Q_\ell$-span of  $V^a_{\ell,\Q,s}$.\\
 
 \noindent  As    $\hbox{\rm Tate}_{\Q_\ell}(X_\eta,i)$ is   invariant under base-change along  finite covers $S'\rightarrow S$ of smooth varieties, to prove  $\hbox{\rm Tate}_{\Q_\ell}(X_\eta,i)\Rightarrow \hbox{\rm WVEt}_{\Q_\ell}(f,i)$ one may first perform such a base-change hence assume:

 \begin{itemize}[leftmargin=*,parsep=0cm,itemsep=0.2cm,topsep=0.2cm ]
 \item $ V_{\Q_\ell,\eta}^a=\Sim[\CH^i(X_{\eta})_{\Q_\ell}\rightarrow\CH^i(X_{\bar \eta})_{\Q_\ell}\stackrel{c_{\ell,\eta}}{\rightarrow} V_{\Q_\ell}]$, which, from the surjectivity of the restriction map $\CH^i(X ) \twoheadrightarrow \CH^i(X_\eta)$, implies  $\Lambda_{\Q_\ell}=V_{\Q_\ell,\eta}^a$;
\item  $\overline{G}_\ell$ is connected - see Footnote \ref{ConnectedGeo}, which ensures
$  V_{\Q_\ell,s}^a\cap (V_{\Q_\ell})^{\overline{G}_\ell}\subset    \widetilde{V}_{\Q_\ell,\eta }  \stackrel{(\alpha)}{=}V_{\Q_\ell,\eta}^a=\Lambda_{\Q_\ell}$,
 where $(\alpha)$   is  $\hbox{\rm Tate}_{\Q_\ell}(X_\eta,i)$.\\
 
 \end{itemize}

    \noindent If $p=0$,  for every  $s_\infty\in S_\infty(\C)$ above  $s\in |S|$,  Artin's comparison isomorphism  yields the following canonical commutative diagram:
    
\begin{equation}\label{Pfuhh}\xymatrix{V_{\Q,s_\infty}^a\cap (V_\Q)^{\overline{G}}\ar[r]^\simeq \ar@{_{(}->}[d]&V_{\ell,\Q,s }^a\cap (V_{\Q_\ell})^{\overline{G}_\ell}\ar@{_{(}->}[d]\\
    \Lambda_\Q\ar[r]_\simeq & \Lambda_{\ell,\Q},}\end{equation}
    which shows   $\hbox{\rm VSing}^0 (f_\infty,i)\Leftrightarrow \hbox{\rm VEt}^0_{\Q_\ell}(f,i)$,  and the isomorphisms   $$(V_{\ell,\Q,s }^a\cap (V_{\Q_\ell})^{\overline{G}_\ell})\otimes_\Q\Q_\ell=V_{\Q_\ell,s }^a\cap (V_{\Q_\ell})^{\overline{G}_\ell},\;\;  \Lambda_{\ell,\Q}\otimes_\Q\Q_\ell=\Lambda_{\Q_\ell},$$ (similar to (\ref{comp})), 
which, together with (\ref{Pfuhh}), show  $\hbox{\rm WVEt}^0_{\Q_\ell}(f,i)\Rightarrow \hbox{\rm VEt}^0_{\Q_\ell}(f,i)$.

 \subsubsection{Proof of  Lemma \ref{MainLem1}}\label{Sec:MainLem1Proof}   As we already  observed that VSing$(f_\infty,i)$ $\Leftrightarrow$  WVEt$_{\Q_\ell}(f,i)$ and WVEt$_{\Q_\ell}(f,i)$ $\Rightarrow$ $\Lambda_{\Q_\ell}=V_{\Q_\ell,s}^a$, $s\in |S|^{\gen}_{\cV_{\Q_\ell}}$ - see Subsection \ref{Sec:ProofLem2(i)}, it only remains to prove that $ c(\Lambda_{\Z_\ell})=0$ for $\ell\gg 0$. This follows at once from  Artin's comparison isomorphism, which  yields the identifications 
 $$  (V_{\Z_\ell }^{\free} /\Lambda_{\Z_\ell })_{\tors} \simeq (V_{\Z }^{ \free} /\Lambda_\Z)_{\tors}\otimes_\Z\Z_\ell.$$
and the fact that  $ (V_{\Z}^{ \free} /\Lambda_\Z)_{\tors}$ is a finite group.

 \subsubsection{Obstruction to the integral Hodge conjecture}\label{Sec:ObstructionIHC} In this subsection, we deduce from Artin's comparison and Theorem \ref{MainTh1} uniform bounds for the obstruction to the integral Hodge conjecture. \\

\noindent   Let $X_\infty$ be a smooth, projective variety over $\C$.
The cycle class map   $$c  :\CH^i(X_\infty) \rightarrow  V_{\Z }:=\SH^{2i}(X_\infty^{\an},\Z(i))$$
for $\Z $-singular cohomology fits into a canonical diagram analogue to (\ref{Diagram:Basic})

$$\xymatrix{\CH^i(X_\infty) \ar[r]\ar@/^2pc/[rrr]^{c} \ar@{->>}[r]\ar[dd]&V_{\Z }^a\ar@{^{(}->}[r]\ar@{->>}[d]&\widetilde{V}_{\Z }\ar@{^{(}->}[r]\ar@{->>}[d]\ar@{}[dr]|\square&V_{\Z }\ar@{->>}[d]\\
 &V_{\Z }^{\free,a}\ar@{^{(}->}[r]\ar@{^{(}->}[d]&\widetilde{V}^{\free}_{\Z }\ar@{^{(}->}[r]\ar@{^{(}->}[d]\ar@{}[dr]|\square&V_{\Z }^{\free}\ar@{^{(}->}[d]\\
\CH^i(X_\infty)_{\Q }\ar@{->>}[r]&V_{\Q }^a\ar@{^{(}->}[r]&\widetilde{V}_{\Q }\ar@{^{(}->}[r]&V_{\Q },}$$
where, writing $G \subset \SGL(V_{\Q })$ for the Mumford-Tate group of the polarizable $\Q$-Hodge structure   $V_{\Q } $ underlies,  $$\widetilde{V}_{\Q }:=(V_{\Q })^{G }$$ is the  $\Q $-vector space of Hodge classes.   The (classical) rational $\Q$-Hodge conjecture in  codimension $i$  for  $X$ \cite{Hodge} 

 $$ \hbox{\rm Hodge}_{\Q}(X_\infty,i) \;\; 
   V_\Q ^a =\widetilde{V}_{\Q }$$
also admits  integral  variants:

 $$\begin{tabular}[t]{lll}
$ \hbox{\rm Hodge}^{\free}_{\Z}(X_\infty,i)$&
   $V_{\Z_\ell}^{\free,a} =\widetilde{V}^{\free}_{\Z}$& (Integral Hodge conjecture modulo torsion);\\
 $\hbox{\rm Hodge}_{\Z}(X_\infty,i)$& 
$V_{\Z}^a =\widetilde{V}_{\Z}$&  (Integral Hodge conjecture).
\end{tabular}$$
\noindent Again, the implications 
 $$\hbox{\rm Hodge}_{\Z}(X_\infty,i)\Rightarrow  \hbox{\rm Hodge}^{\free}_{\Z }(X_\infty,i)\Rightarrow  \hbox{\rm Hodge}_{\Q}(X_\infty,i)$$
 are tautological and, in general, the converse implications are known to fail (see e.g. \cite{AH, G} for examples of the failure of $\ \hbox{\rm Hodge}_{\Q}(X_\infty,i)$ and \cite{Kollar, K1} for  examples of the failure of $\hbox{\rm Hodge}^{\free}_{\Z }(X_\infty,i)$).
  By definition, the obstructions to  $ \hbox{\rm Hodge}_{\Q}(X_\infty,i)$,  $ \hbox{\rm Hodge}^{\free}_{\Z}(X_\infty,i)$, $ \hbox{\rm Hodge}_{\Z }(X_\infty,i)$ are, respectively:

$$\widetilde{C}_{\Q } := \widetilde{V}_{\Q }/V_{\Q }^a,\;\;  \widetilde{C}^{\free}_{\Z } := \widetilde{V}^{\free}_{\Z }/V_{\Z }^{\free,a},\;\;  \widetilde{C}_{\Z }:= \widetilde{V}_{\Z }/V_{\Z }^a,$$
with the properties  that one has the short exact sequence  
\begin{equation}\label{Eq:ObstructionHodge1}0\rightarrow (V_{\Z})_{\tors}/(V_{\Z}^a)_{\tors}  \rightarrow\widetilde{C}_{\Z } \rightarrow  \widetilde{C}^{\free}_{\Z }\rightarrow 0\end{equation}
 and that  $$\hbox{\rm Hodge}_{\Q } \Leftrightarrow (\widetilde{C}^{\free}_{\Z})_{\tors}= \widetilde{C}^{\free}_{\Z} \Leftrightarrow  (\widetilde{C}_{\Z})_{\tors}= \widetilde{C}_{\Z}$$
in which case, (\ref{Eq:ObstructionHodge1}) reads 
$$0\rightarrow (V_{\Z })_{\tors}/(V_{\Z}^a)_{\tors}  \rightarrow(\widetilde{C}_{\Z})_{\tors} \rightarrow  (\widetilde{C}^{\free}_{\Z})_{\tors}\rightarrow 0.$$
 Furthermore,
 $$(\widetilde{C}^{\free}_{\Z})_{\tors}=(C^{\free}_{\Z})_{\tors}:=V^{\free}_{\Z }/V^{\free,a}_{\Z }.$$

\noindent  Assume $p=0$ and fix an embedding $\infty: k\hookrightarrow \C$.  Let  $X$ be a smooth projective variety over $k$.  From the observations in Subsection \ref{Sec:ArtinComp} and   the flatness of $\Z\hookrightarrow \Z_\ell$, Artin's comparison isomorphism   induces the following identifications 
$$  ((V_{\Z})_{\tors}/(V_{\Z}^a)_{\tors})\otimes_\Z\Z_\ell\tilde{\rightarrow} (V_{\Z_\ell})_{\tors}/(V_{\Z_\ell}^a)_{\tors},\;\;  (C_\Z^{\free})_{\tors}\otimes_\Z\Z_\ell\tilde{\rightarrow}(C_{\Z_\ell}^{\free})_{\tors}. $$
As $V_\Z$ is a $\Z$-module of finite type, this shows, in particular,

\begin{enumerate}[leftmargin=*,parsep=0cm,itemsep=0.2cm,topsep=0.2cm,label=\alph*)]
\item $(\widetilde{C}_{\Z_\ell}^{\free})_{\tors} =0$ - hence $(C_{\Z_\ell}^{\free})_{\tors} =0$, for $\ell\gg 0$.
 \item The obstruction $(C_\Z^{\free})_{\tors}$ to Hodge$_\Z^{\free}(X_\infty,i)$ can be recovered from the obstructions $(C_{\Z_\ell}^{\free})_{\tors}$ to Tate$_{\Z_\ell}^{\free}(X,i)$, when $\ell$ varies as 
 $$(C_\Z^{\free})_{\tors}=\oplus_\ell (C_{\Z_\ell}^{\free})_{\tors}.$$
 \end{enumerate}
 
 \noindent As in Subsection \ref{Sec:Statements}, let  now $S$ be a smooth, geometrically connected variety over $k$  and $f:X\rightarrow S$ a smooth projective morphism.   For $s_\infty\in S_\infty(\C)$ above $s\in S$, denote by a subscript $(-)_{s_\infty}$ the various modules attached to $X_{s_\infty}=X_{\infty, s_\infty}$ introduced above (\textit{e.g.} $V_{\Z,s_\infty}:=\SH^{2i}(X^{\an}_{s_\infty},\Z(i))$,  $V^a_{\Z ,s_\infty}:=\hbox{\rm im}[\CH^i(X_{s_\infty })\rightarrow V_{\Z}]$ \textit{etc.}).   Again, one may  investigate how 
 $$\widetilde{\Ob}_{\Z,s}:=|(\widetilde{C}_{\Z,s_\infty})_{\tors}|$$
 vary with $s\in |S|$. A direct consequence of Theorem \ref{MainTh1} and the observations a), b) above is the following. 
 
 \begin{corollaire}\label{Cor:ObstructionIntegralHC} Assume $S$ is a curve and  $ \hbox{\rm VSing}(f_\infty,i)$ holds. Then, for every integer $d\geq 1$, one has 
 $$\widetilde{\Ob}_{\Z}^{\leq d}:=\hbox{\rm sup}\lbrace \widetilde{\Ob}_{\Z,s_\infty}\;|\; s\in |S|^{\leq d}\rbrace <+\infty. $$ 
 \end{corollaire}
 
 \noindent When $i=2$, $(\widetilde{C}_{\Z,s_\infty})_{\tors}$ can again be  described in terms of degree $3$ unramified cohomology. More precisely, set $C_{\Z,s_\infty}:=V_{\Z_\ell}/V_{\Z,s_\infty}^a$. From the short exact sequence 
 $$0\rightarrow  \widetilde{C}_{\Z,s_\infty}\rightarrow C_{\Z,s_\infty} \rightarrow V_{\Z,s_\infty}/\widetilde{V}_{\Z,s_\infty}\rightarrow 0$$
and the fact that $ V_{\Z,s_\infty}/\widetilde{V}_{\Z,s_\infty}$ is torsion-free, one has $(\widetilde{C}_{\Z,s_\infty})_{\tors}=(C_{\Z,s_\infty})_{\tors}$. If $i=2$,   \cite[Thm. 3.7]{CTV} establishes that $(C_{\Z,s_\infty})_{\tors}$ is isomorphic to 
 
$$\SH^3_{nr}(X^\an,\Q /\Z (2))_{\hbox{\rm \tiny ndiv}}\stackrel{def}{=}\hbox{\rm coker}[\SH^3_{nr}(X_\infty^\an,\Z(2))\otimes \Q/\Z\rightarrow \SH^3_{nr}(X_\infty^\an,\Q/\Z(2))].$$

\noindent  Hence Corollary \ref{Cor:ObstructionIntegralHC}  implies (see also \cite[Sec. 5.1]{CTV}):
 
\begin{corollaire}\label{Cor:H3nrHodge} Assume $S$ is a curve and  $\hbox{\rm VSing}(f_\infty,i)$ holds.   Then, for every integer $d\geq 1$, $$\hbox{\rm sup}\lbrace |\SH^3_{nr}(X_{\infty s},\Q /\Z  (2))_{\hbox{\rm \tiny ndiv}}|\; |\;  s\in |S|^{\leq d}  \rbrace | <+\infty.$$  
 \end{corollaire}

 \begin{remarque}
 Using \cite[Thm. 3.11]{CTV} and  Corollary \ref{Cor:ObstructionIntegralHC} for cycles of dimension $1$, one has an analogue of Corollary \ref{Cor:H3nrHodge} with  uniform bounds for the groups $\SH^{n-3}(X_{\infty,s},\mathcal H^{n}_{X_{\infty,s}}(\Q /\Z (n-1)))_{\hbox{\rm \tiny ndiv}}$, where $n=\mathrm{dim}(X_{\infty,s})$.
 \end{remarque}

  \subsection{Crystalline cohomology}\label{Sec:Crys}We now turn to the setting and retain the notation and conventions of Subsection \ref{Sec:MainThm2}.

   \subsubsection{"Comparison" with crystalline cohomology}   A delicate issue   when $p>0$ is to find a suitable analogue of  Artin's comparison isomorphism. Following the strategy of \cite{Emiliano}, this will be achieved by combining Fact  \ref{Emiliano} below, which relies - \textit{via} a $L$-function argument - on the Katz-Messing theorem \cite{KM} and  comparison of various categories of isocrystals, with\footnote{Note that  \cite{Emiliano} was focussed on divisors,  for which the fact that homological and numerical equivalence coincide is known.}   the conjectural statement CrysEt$_{\Q_\ell}(f,i)$. \\
  
  \noindent Let $\mathscr{S}$ be a smooth,   geometrically connected variety over $F$ and  consider a Cartesian square
 $$\xymatrix{\mathcal{X}_{\mathscr{S}}\ar[d]_{f_{\mathscr{S}}}\ar@{}[dr]|\square\ar[r]&\mathcal{X}\ar[d]^f\\
\mathscr{S}\ar[r]&\mathcal{S}.}$$

 \begin{fait}\label{Emiliano}\textit{}\hbox{\rm \cite[Proof of Thm. 1.6.3.1 - esp. (2.1.2.1), Rem. 1.6.3.2]{Emiliano}} Assume the canonical restriction morphism in \'etale $\Q_\ell$-cohomology
 $$\SH^0(\mathcal{S}_{\bar F},R^{2i}f_*\Q_\ell(i))\tilde{\rightarrow}\SH^0(\mathscr{S}_{\bar F},R^{2i}f_*\Q_\ell(i)) $$
 is an isomorphism. Then the canonical restriction morphism in crystalline cohomology
 $$\SH^0(\mathcal{S},R^{2i}f_{\crys,*}\mathcal{O}_{\mathcal{X}/K}) \tilde{\rightarrow}\SH^0(\mathscr{S},R^{2i}f_{\mathscr{S},\crys,*}\mathcal{O}_{\mathcal{X}_{\mathscr{S}}/K})  $$
 is an isomorphism.
  \end{fait}

    \subsubsection{Proof of  Lemma \ref{MainLem2} (ii)} \label{Sec:LemmaBiiproof}  Let   $s\in |S|_{\mathcal{V}_{\ell,\Q_\ell}}^{\gen}$. 
   Recall we are to prove $  V^a_{ \Q_\ell,s}=\Lambda_{\Q_\ell}$.  Replacing $k$, $F$ by finite field extensions, one may assume there exists a   smooth, separated and geometrically connected  scheme $\mathscr{S}$ over $F$ with generic point $\eta_{\mathscr{S}}:\Spec(k(s))\rightarrow \mathscr{S}$ and such that $  \mathscr{S}(F)\not=\emptyset $, and a   Cartesian diagram\\

\begin{equation}\label{Diagram:Crys}
\xymatrix{\mathcal{X}_t\ar[d]_{f_t}\ar[r]\ar@{}[dr]|\square&\mathcal{X}_{\mathscr{S}}\ar[d]_{f_{\mathscr{S}}}\ar@{}[dr]|\square\ar[r]&\mathcal{X}\ar[d]^f&X\ar[l]\ar[d]^f\ar@{}[dl]|\square&X_s\ar[d]^{f_s}\ar@{}[dl]|\square\ar[l] \ar@/_2pc/[lll]\\
F\ar[r]^t\ar[dr]&\mathscr{S}\ar[r]\ar[d]\ar[dr]&\mathcal{S}\ar[d]&S\ar[l]\ar[d]&k(s)\ar[l]_s\ar[dl]\ar@/^2pc/[lll]_{\eta_{\mathscr{S}}}\\
&F&\mathscr{K}\ar[l]&k\ar[l]_{\eta_{\mathscr{K}}}&}\end{equation}
\noindent Replacing further $k$, $F$ by finite field extensions,  one may   assume that 
 
\begin{equation}\label{Eq:Ar1} V^a_{ \Q_\ell, s}=\Sim[\CH^i(X_s) \rightarrow\CH^i(X_{\bar s})\stackrel{c_{\ell,s}}{\rightarrow} V_{\Q_\ell}].\end{equation} 
 From (\ref{Eq:Ar1}), it is enough to show that for every $\widetilde{\alpha}_s\in   \CH^i(X_s)_{\Q}$ with image $\alpha_{\ell,s}:=c_{\ell,s}(\widetilde{\alpha}_s)\in  V_{\Q_\ell }$, there exists $\widetilde{\alpha}\in \CH^i(X )_\Q$ such that $c_{ \ell,s}(\widetilde{\alpha}|_{X_s})=\alpha_{\ell,s}$.  We retain the notation and conventions in Diagram (\ref{Diagram:Crys}).
 Up to shrinking $\mathscr{S}$, one may assume there exists  $\widetilde{\alpha}_{\mathscr{S}}\in \CH^i(\mathcal{X}_{\mathscr{S}})_{\Q}$ such that $\widetilde{\alpha}_{\mathscr{S}}|_{X_s}=\widetilde{\alpha}_s$; write $\widetilde{\alpha}_t:=\widetilde{\alpha}_{\mathscr{S}}|_{\mathcal{X}_t}\in \CH^i(\mathcal{X}_t)_{\Q}$. Consider now the canonical commutative  diagram 
 $$\xymatrix{\CH^i(\mathcal{X} )_{\Q}\ar[d]_{c_{\crys}}\ar[dr]^{|_{\mathcal{X}_t}}\ar[rr]^{|_{\mathcal{X}_{\mathscr{S}}}}&&\CH^i(\mathcal{X}_{\mathscr{S}})_{\Q}\ar[d]^{c_{\crys,\mathscr{S}}}\ar[dl]_{|_{\mathcal{X}_t}}\\
\SH^{2i}_{\crys}(\mathcal{X})  \ar[d]_{\epsilon}\ar[dr]^{|_{\mathcal{X}_t}} &\CH^i(\mathcal{X}_ t)_{\Q}\ar[d]_{c_{\crys,t}}&\SH^{2i}_{\crys}(\mathcal{X}_{\mathscr{S}}) \ar[d]^{\epsilon}\ar[dl]_{|_{\mathcal{X}_t}}\\
\SH^0(\mathcal{S},R^{2i}f_{\crys,*}\mathcal{O}_{\mathcal{X}/K}) \ar@/_2pc/[rr]^\simeq \ar[r]&\SH^{2i}_{\crys}(\mathcal{X}_t) &\SH^0(\mathscr{S},R^{2i}f_{\mathscr{S},\crys,*}\mathcal{O}_{\mathcal{X}_{\mathscr{S}}/K}). \ar[l] }
$$
  As  $s\in S_{\mathcal{V}_{\ell,\Q_\ell}}^{\gen}$, the canonical restriction morphism $$\SH^0(\mathcal{S}_{\bar F},R^{2i}f_*\Q_\ell(i))\tilde{\rightarrow}\SH^0(\mathscr{S}_{\bar F},R^{2i}f_*\Q_\ell(i)) $$
 is an isomorphism - see \cite[\S 2.2.2]{Emiliano}. Here, we implicity use the  reduction \ref{Sec:ConnectedReduction} a), b). Hence, by Fact \ref{Emiliano},  the bottom horizontal arrow is an isomorphism. This implies  that  $\alpha_t:=c_{\crys,t}(\widetilde{\alpha}_t)$ lies in $\SH^0(\mathcal{S},R^{2i}f_{\crys,*}\mathcal{O}_{\mathcal{X}/K}) $. But then, by implication 2) $\Longrightarrow$ 1) in $\hbox{\rm VCrys}(f,i)$, there exists $\widetilde{\alpha}_{\cX}\in \CH^i(\mathcal{X} )_{\Q}$ such that $c_{\crys,t}(\widetilde{\alpha}_{\cX}|_{\mathcal{X}_t})=c_{\crys}(\widetilde{\alpha}_{\cX})|_{\mathcal{X}_t}= \alpha_t=c_{\crys,t}(\widetilde{\alpha}_t)$. By $\hbox{\rm CrysEt}_{\Q_\ell}(f,i)$, this implies $c_{\ell,t}(\widetilde{\alpha}_{\cX}|_{\mathcal{X}_t}) =c_{\ell,t}(\widetilde{\alpha}_t)$. The assertion thus follows, with $\widetilde{\alpha} =\widetilde{\alpha}_{\cX}|_X$, from the  canonical commutative specialization diagram of cycle class maps 
$$\xymatrix{ &\CH^i(\mathcal{X})_{\Q}\ar[d]^{|_{X_s}}\ar[dl]_{|_{\mathcal{X}_t}}\ar[dr]^{|_{X}}& \\
\CH^i(\mathcal{X}_t)_{\Q}\ar[d]_{c_{\ell,t}}&\CH^i(X_s)_{\Q}\ar[d]^{c_{\ell,s}}\ar[l]_{sp_{s,t}} & \CH^i(X )_{\Q} \ar[l]_{|_{X_s}}  \\
\SH^{2i}(\mathcal{X}_{\bar t},\Q_\ell(i))\ar@{=}[r]&\SH^{2i}(X_{\bar s},\Q_\ell(i)).& }$$

\end{document}